\documentclass[12pt]{article}

 \usepackage{hyperref}
 \usepackage{amssymb}
\usepackage{graphicx}
\graphicspath{ {./research project/} }
\usepackage{caption}
\usepackage{mathtools}
\usepackage{enumitem}
\usepackage{amsthm}
\usepackage{fancyhdr}

\newtheorem{remark}{Remark}
\newtheorem{theorem}{Theorem}
\newtheorem{corollary}{Corollary}

\title{}
\author{}
\date{}
\usepackage{epigraph}
\usepackage{tkz-euclide}
\usepackage{comment}
\usepackage{xcolor}
\usepackage{hyperref}
\usepackage{multirow}
\usepackage{setspace}

\author{ }
\date{}

\begin{document}

\maketitle

{\centering\Large \bf New Identities for the Family of Zeta Function by Using
Distributional Representations}\\
\vspace{0.5cm}
Asghar Qadir$^*$ and Aamina Jamshaid$^\dag$\\
$^*$Pakistan Academy of Sciences\\
3 Constitution Avenue, G-5/2, Islamabad 44000, Pakistan\\
asgharqadir46@gmail.com\\
$^\dag$Abdus Salam School of Mathematical Sciences\\
Government College University \\
68-B New Muslim Town, Lahore, Pakistan\\
aaminajamshaid@gmail.com\\
Current affiliation, College of E\&ME, NUST, Islamabad

\vspace{0.5cm}

Chaudhry and Qadir obtained new identities for the gamma
function by using a ``distributional representation'' for it. Here we obtain
new identities for the Riemann zeta function and its family by using that
representation for them. This also leads to new identities involving the
Dirichlet eta and Lambda functions.

\section{Introduction}

Chaudhry and Qadir defined the Fourier transform representation (FTR) and
distributional representations (DR) for gamma functions and obtained new
identities for them \cite{dgamma}. Earlier, Chaudhry and Zubair had defined
the ``extended gamma function'' \cite{gamma},
\begin{equation}
\Gamma_b(s)=\int_0^{\infty}t^{s-1}e^{-t-b/t}dt,\; (b > 0;~ b=0,~\sigma>0)~,
\end{equation}
where $s=\sigma+\iota\tau$, which reduces to the gamma function for $b=0$.
Distributions are functionals, which are defined only in terms of an inner
product with sufficiently well-behaved functions including the gamma and zeta
functions. In particular, they used the Dirac delta functions, which is
normally defined in the real domain. However, it was used for complex
variables, assuming that this could be done. The same assumption had been
made by Lighthill \cite{lighthill}, but was made rigorous by Qadir and
Tassadiq \cite{ggamma} and Tassadiq \cite{thesis}, who defined the FTR and
DRs for this function\cite{ggamma}. Chaudhry et al. \cite{erzf} used the
method of extension for the Riemann zeta function. In fact it could be
applied to the full family of zeta functions \cite{edwards}. The question
naturally arose whether the FTR and DR could be defined for the Riemann zeta
function and its family, and whether the DR could be used to obtain new
identities for them. The former question was addressed by Jamshaid
\cite{athesis} and it was found that the DR could be defined for some, but
not all of them. Here we address the latter question and find a number of
identities that are nontrivial, including two for the Dirichlet eta
\cite{milgram} and lambda functions  \cite{roy}. It will be seen that the new
representation is defined in terms of an inner product.

The DR of the gamma function is obtained from its FTR
\begin{equation}
\Gamma(s)=\mathcal{F}[e^{\sigma x}e^{-e^{x}}; \tau]~,
\end{equation}
where $\mathcal{F}$ is the Fourier transform, the first term in the bracket
is the function on which it acts and $\tau$ is the new variable of the
function which then defines the complex variable, $s$. Using the property of
the FTR \cite{zemanian}
\begin{equation}
\mathcal{F}[e^{\sigma x};\tau]=2\pi\delta(\tau-\iota\sigma),
\end{equation}
yields an infinite series of Dirac-delta functions
\begin{equation}
\Gamma(\sigma+\iota\tau)=\mathcal{F}[f_{\sigma}(x);\tau]=2\pi\sum_{n=0}^{\infty}
\frac{(-1)^n}{n!}\delta(\tau-\iota(\sigma+n))~,
\end{equation}
where $f_{\sigma}(x)=e^{\sigma x}e^{-e^{x}}$, and so when the FTR of the
extension
\begin{equation}
\Gamma_b(\sigma+\iota\tau)=\mathcal{F}[e^{\sigma x}e^{-e^x-be^{-x}};\tau]
\end{equation}
is used, it gives the DR
\begin{equation}
\Gamma_b(\sigma+\iota\tau)=2\pi\sum_{n=0}^{\infty}\sum_{m=0}^{\infty}
\frac{(-1)^{n+m}b^{n}}{m!n!}\delta(\tau-\iota(\sigma+m-n)).
\end{equation}
It must be borne in mind that the delta function is a distribution, and is
only meaningful in terms of an inner product with a function from a class of
functions that is suitably well-behaved, like the gamma function and the zeta
functions. However, one must further ensure that the series converges. This
was done for the Bose-Einstein, the Fermi-Dirac, and their extensions
\cite{thesis}.

The plan of this paper is as follows. The family of the zeta function is
introduced in the next section. The DR of the zeta family is given in Section
3. In the subsequent section we obtain some identities for the integrals of
inner products of the Riemann zeta function (RZF) with other functions using
the DR. These include, also, the Dirichlet eta and Lambda functions.

\section{The Zeta Family and its Extensions}

The integral representation of the {\it Riemann zeta function} (RZF) is
\cite{edwards}
\begin{equation}
\zeta(s)=\int_{0}^{\infty}t^{s-1} (1-e^t)^{-1} e^{-t}dt~,\;(Re(s)>1)
\end{equation}
and for its extension (the eRZF) is \cite{erzf},
\begin{equation}
\zeta_b(s)=\int_0^{\infty}t^{s-1}(1-e^t)^{-1}e^{-t-b/t}dt~,
~(b >0~;~b=0~,~\sigma>1)~.
\end{equation}
The {\it Hurwitz zeta function} (HZF) is defined by the integral
representation
\cite{edwards},
\begin{equation}
\zeta(s,a)=\frac{1}{\Gamma(s)}\int_0^{\infty}\frac{t^{s-1}e^{-at}}{1-e^t}dt
,~(\sigma>1,~Re(a)>0),
\end{equation}
which reduces to the RZF for $a=1$. For its extension (eHZF) it is
\begin{equation}
\zeta_b(s,a)=\frac{1}{\Gamma(s)}\int_0^{\infty}\frac{t^{s-1} e^{-at}}
{1-e^{-t}}e^{-b/t}dt~,~(b>0~;~b=0~,\sigma>1~,~Re(a)>0).
\end{equation}
Similarly, the integral representation defines the {\it Hurwitz-Lerch zeta
$($HLZF$)$ function}
\begin{equation}
\Phi(z,s,a)=\frac{1}{\Gamma(s)}\int_0^{\infty}\frac{t^{s-1}e^{-at}}
{1-ze^{-t}}dt~,
\end{equation}
where $Re(a)>0$ and either $|z|\leq1,z\neq1,\sigma>0$, or $z=1$, $\sigma>1$,
which reduces to the earlier functions by appropriate choice of $a$  and
restriction of the domain of $z$. Its extension (eHLZF) \cite{athesis} has
the representation
\begin{equation}
\Phi(z,s,a;b)=\frac{1}{\Gamma(s)}\int_0^{\infty}\frac{t^{s-1}e^{-at}}
{1-ze^{-t}}e^{-b/t}dt~,~(b>0)~.
\end{equation}
For $b=0$ we require $Re(a)>0,~\sigma>0,|z|<1$, or if $z=1,~\sigma>1$. We
are not giving other members of the zeta family, as they have not provided
any new identities.

There are two other functions which are only slightly different from the RZF
but not considered members of the family: the {\it Dirichlet eta function}
\cite{milgram}
\begin{equation}
\eta(s)=(1-2^{1-s})\zeta(s)~,
\end{equation}
and the {\it Dirichlet Lambda function} \cite{roy}
\begin{equation}
\Lambda(s)=2(1-2^{1-s})\zeta(s)~.
\end{equation}

\section{The DRs Using the FTRs of the Family}

To obtain the DRs of functions one needs their FTRs and then uses Eq.(3). The
most general form of the zeta family is the eHLZF, in that all the others can
be hierarchically obtained from there by suitable choice of the parameters
involved (as given above). As such, one only needs to give the FTR for it and
leave it to the reader to obtain the others that are used for the DR.

The FTR of the eHLZF is
\begin{equation}
\Phi(z,\sigma+\iota\tau,a;b)=\frac{1}{\Gamma(\sigma+\iota\tau)}\mathcal{F}
[v_{a\sigma}^{b}(x,z);\tau]\;\;(b>0)~,
\end{equation}
where
\begin{equation}
v_{a\sigma}^{b}(x,z):=\frac{e^{\sigma x}e^{-ae^x-be^{-x}}}{1-ze^{-e^x}}~.
\end{equation}
Replacing $t$ by $e^x$, the expression $e^{i\tau x}$ leads to the FT. For
$b=0$ in Eq.(16) we get the FTR of the HLZF,
\begin{equation}
\Phi(z,\sigma+\iota\tau,a)=\frac{1}{\Gamma(\sigma+\iota\tau)}
F[v_{a\sigma}(x,z);\tau],
\end{equation}
where
\begin{equation}
v_{a\sigma}(x,z)=:\frac{e^{\sigma x} e^{-ae^x}}{1-ze^{-e^x}}~,
\end{equation}
and for $b=0$ we require $Re(a)>0,\sigma>0,|z|<1$, or if $z=1,~\sigma>1$.

Taking $z=1$ in Eq.(15), gives the FTR of the eHZF
\begin{equation}
\zeta_{b}(\sigma +\iota\tau, q)=\frac{1}{\Gamma(\sigma+\iota\tau)}
\mathcal{F}[h_{a\sigma}^{b}(x);\;\tau],\;
\end{equation}
where
\begin{equation}
h_{a\sigma}^{b}(x):=\frac{e^{\sigma x} e^{-ae^x-be^{-x}}}{1-e^{-e^x}}~.
\end{equation}
Now also putting $b=0$ in Eq.(19), yields
\begin{equation}
\zeta(\sigma +\iota\tau, q)=\frac{1}{\Gamma(\sigma +\iota\tau)}
\mathcal{F}[h_{a\sigma}(x); \tau],
\end{equation}
where
\begin{equation}
h_{\sigma}(x):= \frac{e^{\sigma x-ae^{x}}}{1-e^{-e^{x}}}~~~
(\sigma>1,Re(a)>0),
\end{equation}
which is the FTR of the HZF. Setting $z=1$ and $a=1$ in Eq.(15) gives the FTR
of the eRZF
\begin{equation}
\zeta_b(\sigma+\iota\tau)=\frac{1}{\Gamma(\sigma+\iota\tau)}
\mathcal{F}\:[g_\sigma^{b}(x);\tau],~
\end{equation}
where
\begin{equation}
g_\sigma^{b}(x):=\frac{e^{\sigma x-be^{-x}-e^{x}}}{1-e^{-e^{x}}}~~~(b > 0),
\end{equation}
for $b=0$, $\sigma>1$. Finally putting $b=0$ in Eq.(21), the FTR of the RZF
is obtained,
\begin{equation}
\zeta(\sigma+\iota\tau)=\frac{1}{\Gamma(\sigma +\iota\tau)}\mathcal
{F}[g_{\sigma}(x);\tau],
\end{equation}
where
\begin{equation}
g_\sigma(x):=\frac{e^{\sigma x-e^{x}}}{1-e^{-e^{x}}}~~~(\sigma>1).
\end{equation}

The DR of the RZF is
\begin{equation}
\zeta(\sigma+\iota\tau)=\frac{2\pi}{\Gamma(\sigma+\iota\tau)}
\sum_{l=0}^{\infty}\sum_{n=0}^{\infty}\sum_{m=0}^{\infty}(-1)^{l}
\frac{(-n)^m}{l!m!}\delta(\tau-\iota(\sigma+l+m))~(\sigma>1)~.
\end{equation}
Using $\delta(\alpha s)=\frac{1}{|\alpha|}\delta(s)$, we can write the DR in
the alternate form,
\begin{equation}
\zeta(\sigma+\iota\tau)=\frac{2\pi}{\Gamma(\sigma+\iota\tau)}
\sum_{l=0}^{\infty}\sum_{n=0}^{\infty}\sum_{m=0}^{\infty}(-1)^{l}
\frac{(-n)^m}{l!m!}\delta(\iota\tau+(\sigma+l+m)).
\end{equation}

The DR of the eRZF is,
\begin{equation}
\zeta_b(\sigma+\iota\tau)=\frac{2\pi}{\Gamma(\sigma+\iota\tau)}
\sum_{n=0}^{\infty}\sum_{k=0}^{\infty}\sum_{m=0}^{\infty}
\sum_{l=0}^{\infty}(-1)^{k}\frac{(-n)^m (-b)^l }{k!l!m!}
\delta(\tau-\iota(\sigma+k-l+m))~,
\end{equation}
for $b>0$ and if $b=0$ then $\sigma>1$.

The DR of the HZF is given by
\begin{equation}
\zeta(\sigma+\iota\tau,q)=\frac{2\pi}{\Gamma(\sigma+\iota\tau)}
\sum_{n=0}^{\infty}\sum_{m=0}^{\infty}\sum_{k=0}^{\infty}
\frac{(-n)^k(-q)^m}{n!k!}\delta(\tau-\iota(\sigma+k+m))~,
\end{equation}
where $\sigma>1,~Re(a)>0$, and the DR of the eHZF is given by
\begin{equation}
\zeta_{b}(\sigma+i \tau,q)=\frac{2\pi}{\Gamma(\sigma+i \tau)}\sum_{n=0}^{\infty}\sum_{l=0}^{\infty}\sum_{m=0}^{\infty}\sum_{k=0}^{\infty}\frac{(-n)^k (-b)^l (-q)^m}{n!\:l!\:k!}\delta(\tau-i(\sigma+k-l+m))
\end{equation}
for $b>0$ and if $b=0$ then $\sigma>1$.

The HLZF has the DR,
\begin{equation}
\Phi(z,s,a)=\frac{2\pi}{\Gamma(\sigma+\iota\tau)}\sum_{n=0}^{\infty}
\sum_{m=0}^{\infty}\sum_{k=0}^{\infty}
\frac{(-a)^n (-z)^k m^k}{n!k!}\delta(\tau-\iota(\sigma+n+k))~,
\end{equation}
with the corresponding constraints, and the eHLZF has the DR\\
$\Phi(z,s,a;b)=$
\begin{equation}
~~~\frac{2\pi}{\Gamma(\sigma+\iota\tau)}\sum_{m=0}^{\infty}
\sum_{n=0}^{\infty}\sum_{k=0}^{\infty}\sum_{m=0}^{\infty}
\frac{(-a)^m (-b)^l z^n (-n)^k}{m!k!l!}\delta(\tau-\iota(\sigma+n+k-l))~.
\end{equation}

\section{Identities Using DR of the RZF}

The DR of the RZF is only meaningful when defined as an inner product of the
RZF with some element of the space of test function. Since there is a gamma
function in the denominator of Eq. (27) multiplying the series of delta
functions, the space for which the series converges is defined must be all
well-behaved multiples of the gamma function, i.e. $\Gamma(s)\phi(s)$. Then,
\begin{equation}
\langle\zeta(s),\Gamma(s)\phi(s)\rangle=2\pi\sum_{l=0}^{\infty}
\sum_{n=0}^{\infty}\sum_{m=0}^{\infty}(-1)^{l}\frac{(-n)^m}{l!m!}\phi(-(l+m)).
\end{equation}
This general formula will be used to obtain a number of identities for the
RZF and its family.

The simplest choice of $\phi(s)$ is unity. This leads to the identity.
\begin{theorem}
The series of RZFs for negative integer arguments is
\begin{equation}
\sum_{m=0}^{\infty}\frac{(-1)^{m}}{m!}\zeta(-m)=\frac{1}{e-1}.
\end{equation}
\end{theorem}
\begin{proof}
Setting $\phi(s)=1$ in Eq.(33), gives\\
$\langle\zeta(s),\Gamma(s)\rangle=$
\begin{equation}
~~~2\pi\sum_{l=0}^{\infty}\sum_{n=0}^{\infty}\sum_{m=0}^{\infty}(-1)^{l+m}
\frac{n^m}{l!m!}\int_{-\infty}^{\infty}\delta(s+l+m)ds=\frac{2\pi}{e-1}.
\end{equation}
But in \cite{ggamma}, the DR of the gamma function yields
\begin{equation}
\langle\Gamma(s),\zeta(s)\rangle=2\pi\sum_{m=0}^{\infty}
\frac{(-1)^{m}}{m!}\zeta(-m)~,
\end{equation}
which completes the proof.
\end{proof}

The above theorem can easily be extended by the Chaudhry-Zubair method:
\begin{theorem}
For the eRZF
\begin{equation}
\sum_{m=0}^{\infty}\frac{(-1)^{m}}{m!}\zeta_{b}(-m)=
2\pi\frac{e^{-b}}{e-1}~.
\end{equation}
\end{theorem}
\begin{proof}
Taking the inner product of the DR of the eRZF and the gamma function
$\Gamma(s)$, for $\phi(s)=1$, one obtains
\begin{eqnarray}
\left <\zeta_{b}(s),\Gamma(s) \right> &=& 2\pi\sum_{m=0}^{\infty}\sum_{k=0}^{\infty} \sum_{n=0}^{\infty}\sum_{l=0}^{\infty}(-1)^{k}\frac{(-n)^{l}(-b)^m}{k!m!l!} \nonumber\\
&=& 2\pi\sum_{m=0}^{\infty}\frac{b^m}{m!}\sum_{k=0}^{\infty}\frac{(-1)^{k}}{k!}
\sum_{n=0}^{\infty}\sum_{l=0}^{\infty}(-1)^{l}\frac{n^{l}}{l!}~,
\end{eqnarray}
which yields
\begin{equation}
\left <\zeta_{b}(s),\Gamma(s) \right>= 2\pi\frac{e^{-b}}{e-1}~.
\end{equation}
Instead of using the DR of the eRZF with the gamma function, one can use the
DR of the gamma function for the RZF as the test function, to obtain
\begin{equation}
\left<\Gamma(s),\zeta_{b}(s)\right>=2\pi\sum_{m=0}^{\infty}
\frac{(-1)^{m}}{m!}\zeta_{b}(-m)~,
\end{equation}
which directly provides the result.
\end{proof}

\begin{theorem}
For the HZF
\begin{equation}
\sum_{m=0}^{\infty}\frac{(-1)^{m}}{m!}\zeta(-m, q)=\frac{e^{1-q}}{e-1}~.
\end{equation}
\end{theorem}
\begin{proof}
Again, take the inner product of the DR of the HZF with the gamma function,
\begin{eqnarray}
\left <\zeta(s,q),\Gamma(s)\right> &=& 2\pi\sum_{n=0}^{\infty}\sum_{m=0}^{\infty}
\sum_{k=0}^{\infty}\frac{(-m)^{k}(-q)^n}{k!n!} \nonumber \\
&=& 2\pi\sum_{n=0}^{\infty}\frac{(-1)^{n}}{n!}q^{n}\sum_{m=0}^{\infty}
\sum_{k=0}^{\infty}(-1)^{k}\frac{m^k}{k!}~,
\end{eqnarray}
from which one obtains
\begin{equation}
\left<\zeta(s,q),\Gamma(s)\right>=\frac{2\pi e^{1-q}}{e-1}~.
\end{equation}
Again, reversing the order of the DR to the gamma function,
\begin{equation}
\left <\Gamma(s),\zeta(s,q)\right>=2\pi\sum_{m=0}^{\infty}
\frac{(-1)^{m}}{m!}\zeta(-m,q)~,
\end{equation}
which provides the required proof.
\end{proof}

Following the above procedures for the eHZF, one easily gets:
\begin{theorem}
\begin{equation}
\sum_{m=0}^{\infty}\frac{(-1)^{m}}{m!}\zeta_{b}(-m, q)=\frac{e^{1-(q+b)}}{e-1}~,
\end{equation}
\end{theorem}

\begin{theorem}
For the Dirichlet eta function,
\begin{equation}
\sum_{m=0}^{\infty}\frac{(-1)^{m}}{m!}\eta(-m)=\frac{1}{e+1}~.
\end{equation}
\end{theorem}
\begin{proof}
Taking $\phi=(1-e^{1-s})$ in Eq. (34), for the inner product of the DR of the
zeta function with the gamma function, after some manipulation the RHS can be
reduced to $2\pi/(e+1)$. Now, one can commute the zeta and gamma functions in
Eq. (34), and thus use the DR of the gamma function, with the product of the
zeta with the chosen $\phi$. But this is simply the eta function, which in
turn is $2\pi$ times the series in the RHS of Eq. (47). This proves the
theorem.
\end{proof}
\begin{theorem}
Similarly, for the Dirichlet Lambda function,
\begin{equation}
\sum_{m=0}^{\infty}\frac{(-1)^{m}}{m!}\Lambda(-m)=\frac{e}{e^{2}-1}~.
\end{equation}
\end{theorem}

\section{Concluding Remarks}

We applied the DR to the family of the zeta function as it had been applied
to the gamma and extended gamma functions. Previously they had proved useful
to obtain new identities for the gamma functions. Here we obtained many new
identities involving the RZF and its family, including the Dirichlet, eta and
Lambda functions. We also found new identities involving the Bernoulli
numbers. It will be worthwhile to look for DRs of other special functions
that have not been investigated here. Further it would be worth exploring if
one can obtain DRs of functions of two (or more) variables. To conclude,
there had earlier been a question of whether the DR could do anything more
than the FTR did. We have seen that it provides a whole slew of new
identities and opens up a number of new possibilities. It is definitely worth
studying further.

\section*{Acknowledgments}

We are very grateful to Gabor Korvin for very useful comments. Aamina
Jamshaid is most grateful to the Abdus Salam School of Mathematical Sciences
for financial support during her M.Phil.


\end{document}